\documentclass[a4paper,10pt]{article}
\usepackage{amsmath}
\usepackage{amsfonts}
\usepackage{amssymb}
\usepackage{amsthm}
\usepackage{enumitem}

\theoremstyle{plain}
\newtheorem{theorem}{Theorem}
\newtheorem{lemma}[theorem]{Lemma}
\newtheorem{corollary}[theorem]{Corollary}
\newtheorem{question}[theorem]{Question}

\theoremstyle{remark}

\newcommand{\lem}[1]{Lemma~\ref{#1}}
\newcommand{\thm}[1]{Theorem~\ref{#1}}
\newcommand{\cor}[1]{Corollary~\ref{#1}}
\newcommand{\sect}[1]{Section~\ref{#1}}

\newcommand{\defn}{\emph}
\newcommand{\lc}{\leq}
\newcommand{\genBy}{\prec}
\newcommand{\A}{\mathcal{A}}
\newcommand{\B}{\mathcal{B}}
\newcommand{\C}{\mathcal{C}}
\newcommand{\D}{\mathcal{D}}
\newcommand{\G}{\mathcal{G}}
\newcommand{\sta}{\mathcal{S}}
\newcommand{\HM}{\mathcal{T}}
\newcommand{\gen}[3]{\mathcal{F}(#1, #2, #3)}
\newcommand{\genX}[4]{\mathcal{F}(#1, #2, #3)(#4)}

\newcommand{\forsome}{\text{ for some }}
\newcommand{\andd}{\text{ and }}

\title{Maximum hitting for $n$ sufficiently large}
\author{Ben Barber\footnote{Department of Pure Mathematics and Mathematical Statistics, Centre for Mathematical
Sciences, Wilberforce Road, Cambridge, CB3 0WB, UK.  {\tt b.a.barber@dpmms.cam.ac.uk}}}

\begin{document}

\maketitle

\begin{abstract}
For a left-compressed intersecting family $\A \subseteq [n]^{(r)}$ and a set $X \subseteq [n]$, let $\A(X) = \{A \in \A : A \cap X \neq \emptyset\}$.  Borg asked: for which $X$ is $|\A(X)|$ maximised by taking $\A$ to be all $r$-sets containing the element 1?  We determine exactly which $X$ have this property, for $n$ sufficiently large depending on $r$.
\end{abstract}

\section{Introduction}

Write $[n]=\{1,2,\ldots,n\}$ and $[m,n]=\{m, m+1,\ldots,n\}$.  Denote the set of $r$-sets from a set $S$ by $S^{(r)}$.  A \defn{family} of sets is a subset of $[n]^{(r)}$ for some $n$ and $r$.  We think of a set $A$ as an increasing sequence of elements $a_1a_2\ldots a_r$.  The \defn{compression order} on $[n]^{(r)}$ has $A \lc B$ if and only if $a_i \leq b_i$ for $1 \leq i \leq r$.  A family $\A$ is \defn{left-compressed} if $A \in \A$ whenever $A \leq B$ for some $B \in \A$.  The corresponding notion of left-compression is described in \sect{comp-gen}.

We call a family \defn{intersecting} if $A \cap B \neq \emptyset$ for all $A, B \in \A$.  (If $n<2r$ then every family is intersecting.)  The most basic result about intersecting families is the Erd\H{o}s-Ko-Rado Theorem.  For any $n$ and $r$, write $\sta = \{A \in [n]^{(r)}:1\in A\}$ for the \defn{star} at 1.

\begin{theorem}[Erd\H{o}s-Ko-Rado \cite{ErdosKoRado}]\label{EKR}
 If $n \geq 2r$ and $\A \subseteq [n]^{(r)}$ is intersecting, then $|\A| \leq |\sta|$.
\end{theorem}

Borg considered a variant problem where we only count members that meet some fixed set $X$.  For a family $\A$ and a non-empty set $X$, write
\[
 \A(X) = \{A \in \A:A \cap X \neq \emptyset\}.
\]
\thm{EKR} tells us that we can maximise $|\A(X)|$ by taking $\A$ to consist of all $r$-sets containing some fixed element of $X$.  To avoid this trivial case we insist that $\A$ be left-compressed, which rules out stars centred anywhere but $1$.  The star at $1$ remains the optimal family if $1 \in X$, so we assume further that $X \subseteq [2,n]$.

\begin{question}
 For which $X$ do we have $|\A(X)| \leq |\sta(X)|$ for all left-compressed intersecting families $\A$?
\end{question}

Borg asked this question in \cite{Borg}, giving a complete answer for the case $|X| \geq r$ and a partial answer for the case $|X| < r$.  Call $X$ \defn{good} (for $n$ and $r$) if for every left-compressed intersecting family $\A\subseteq [n]^{(r)}$ we have $|\A(X)| \leq |\sta(X)|$.

\begin{theorem}[Borg \cite{Borg}]\label{borg}
Let $r \geq 2$, $n \geq 2r$ and $X \subseteq [2,n]$.
 \begin{enumerate}[topsep=4pt,itemsep=0pt,label={\normalfont (\alph*)}]
  \item If $|X| > r$, then $X$ is good.
  \item If $X$ is good and $X \leq X'$, then $X'$ is good.
  \item For any $k \leq r$, $\{2k, 2k+2, \ldots, 2r\}$ is good.
  \item If $n=2r$ and $|X|=r$, then $X$ is good if and only if $\{2,4,\ldots,2r\} \lc X$.
  \item If $n > 2r$, $|X| = r$ and either
  \begin{enumerate}[topsep=0pt,itemsep=0pt,label={\normalfont (\roman*)}]
   \item $r \geq 4$ and $X \neq [2,r+1]$,
   \item $r = 3$ and $\{2,3\} \not\subseteq X$, or
   \item $r = 2$ and $\{2,3\} \neq X$,
  \end{enumerate}
  then $X$ is good.  Otherwise, $X$ is not good.
 \end{enumerate}
\end{theorem}

It is not true that all $X$ are good.  For example, consider the \defn{Hilton-Milner} family $\HM = \sta([2,r+1]) \cup \{[2,r+1]\}$ .  The family $\HM$ is left-compressed and for any $X \subseteq [2,r+1]$, $|\HM(X)| = |\sta(X)| + 1$, so $X$ is not good.

Our main result is that, surprisingly, for large $n$ and $|X| \geq 4$ this turns out to be the only obstruction.

\begin{theorem}\label{main}
 Let $r \geq 3$, $n \geq 2r$ and $X \subseteq [2,n]$ with $|X| \leq r$.  If $X \not\subseteq [2,r+1]$ and either
 \begin{enumerate}[topsep=4pt,itemsep=0pt,label={\normalfont (\roman*)}]
  \item $|X| \geq 4$,
  \item $|X| = 3$ and $\{2,3\}\not\subseteq X$,
  \item $|X| = 2$ and $2, 3 \not\in X$, or
  \item $|X| = 1$,
 \end{enumerate}
 then, for $n$ sufficiently large, $X$ is good.  Otherwise, $X$ is not good.
\end{theorem}

For $r=2$, condition (iii) needs to be replaced by $X \neq \{2, 3\}$.  The result can then be checked easily by hand or read out of \thm{borg} in conjunction with the Hilton-Milner example, so we assume $r\geq 3$ for simplicity.

Our proof uses Ahlswede and Khachatrian's notion of generating sets to express  the sizes of maximal left-compressed intersecting families, and their restrictions under $X$, as polynomials in $n$.  It turns out to be sufficient to consider only leading terms, reducing a question about intersecting families of $r$-sets to a question about intersecting families of 2-sets, which have a very simple structure.

\sect{comp-gen} sets out the basic properties of compressions and generating sets that we shall use.  \sect{mlcif} describes a way of thinking about maximal left-compressed intersecting families and proves the lemma that allows us to compare coefficients of polynomials instead of set sizes.  \sect{main-theorem} completes the proof of \thm{main}.  \sect{generalisations} discusses possible improvements and generalisations.

%%%%%%%%%%%%%%%%%%%%%%%%%%%%%%%%%%%%%%%%%%%%%%%%%%%%%%%%%%%%%%%%%%%%%%%%%%%%%%%%%%%%%%%%%%%%%

\section{Compressions and generating sets}\label{comp-gen}

In this section we describe the notion of left-compression corresponding to $\lc$ on $[n]^{(r)}$ and the use of generating sets.

\subsection{Compressions}

For a set $A$, and $i<j$, the \defn{$ij$-compression} of $A$ is
\[
 C_{ij}(A) = \begin{cases}
            A - j + i & \text{if }j \in A, i \not\in A,
\\          A         & \text{otherwise;}
           \end{cases}
\]
that is, replace $j$ by $i$ if possible.  Observe that $A\lc B$ if and only if $A$ can be obtained from $B$ by a sequence of $ij$-compressions.

For a set family $\A$, define
\[
 C_{ij}(\A) = \{C_{ij}(A) : A \in \A \andd C_{ij}(A) \not\in \A\} \cup \{A : A \in \A \andd C_{ij}(A) \in \A\};
\]
that is, compress $A$ if possible.  Observe that $\A$ is left-compressed if and only if $C_{ij}(\A)=\A$ for all $i<j$.  We will use the following basic result.
\begin{lemma}\label{compression}
 If $\A$ is intersecting then $C_{ij}(\A)$ is intersecting.
\end{lemma}
\begin{proof}
 The proof is an easy case check.  Details, and a further introduction to compressions, can be found in Frankl's survey article \cite{Frankl}.
\end{proof}
\lem{compression} means that we can always compress an intersecting family to a left-compressed intersecting family of the same size by repeatedly applying $ij$-compressions.  We eventually reach a left-compressed family as $\sum_{A \in \A} \sum_{i=1}^r a_i$ is positive and strictly decreases with each successful compression.

\subsection{Generating sets}

For any $r$ and $n$, and a collection $\G$ of sets, the family \defn{generated} by $\G$ is
\[
 \gen{r}{n}{\G} = \{A \in [n]^{(r)} : A \supseteq G \forsome G \in \G\}.
\]
Generating sets were introduced by Ahlswede and Khachatrian \cite{AhlswedeKhachatrian}, and are useful for the study of intersecting families because they give a restricted number of sets on which all the intersecting actually happens.

\begin{lemma}[\cite{AhlswedeKhachatrian}]\label{old-generators}
 For $n \geq 2r$, $\gen{r}{n}{\G}$ is intersecting if and only if $\G$ is.
\end{lemma}
\begin{proof}
 If $\G$ is intersecting then certainly $\gen{r}{n}{\G}$ is.  Conversely, if $\G$ contains two disjoint sets then (since $n \geq 2r$) they can be completed to disjoint $r$-sets in $\gen{r}{n}{\G}$.
\end{proof}

If $\G$ generates a left-compressed intersecting family then 
\[
 \G' = \{G' : G' \leq G \forsome G \in \G\}
\]
generates the same family, so we may assume that $\G$ is `left-compressed' (overlooking non-uniformity) and can therefore be described by listing its maximal elements.  It is convenient to take
\[
 \gen{r}{n}{\G} = \{A \in [n]^{(r)} : A \genBy G \forsome G \in \G\},
\]
where $A \genBy G$ (`$A$ is generated by $G$') if and only if $|G| \leq |A|$ and $a_i \leq g_i$ for $1 \leq i \leq |G|$.  We can think of $\genBy$ as an extension of $\lc$ to the non-uniform case, where `missing' elements are assumed to take the value $\infty$.  Thus
\begin{align*}
 123 & \genBy 12\; (\;= 12\infty); \\
 (12\infty =\;)\; 12 & \not\genBy 123.
\end{align*}
The following weaker form of \lem{old-generators} is better suited to our new definition and is sufficient for our purposes.

\begin{corollary}\label{new-generators}
Let $n \geq 2r$ and $\G$ be a collection of subsets of $[2s]$ of size at most $s$.  If $\gen{s}{2s}{\G}$ is intersecting, then so is $\gen{r}{n}{\G}$. \qed
\end{corollary}

%%%%%%%%%%%%%%%%%%%%%%%%%%%%%%%%%%%%%%%%%%%%%%%%%%%%%%%%%%%%%%%%%%%%%%%%%%%%%%%%%%%%%%%%%%%%%%%%%%%%%%%

\section{Maximal left-compressed intersecting families}\label{mlcif}

We say an intersecting family $\A \subseteq [n]^{(r)}$ is \defn{maximal} if no other set can be added to $\A$ while preserving the intersecting property.  The maximal objects in the set of left-compressed intersecting families are maximal intersecting families (otherwise an extension could be compressed to a left-compressed extension), so the ordering of `maximal' and `left-compressed' is unimportant.

The maximal left-compressed intersecting subfamilies of $[n]^{(2)}$ are $\{12, 13, \ldots, 1n\}$ and $\{12, 13, 23\}$, and we can already distinguish between these families when $n=4$.  In fact, the same phenomenon occurs for all $r$.

\begin{lemma}\label{description}
 Let $\A \subseteq [2r]^{(r)}$ be a maximal left-compressed intersecting family and $n \geq 2r$.  Then $\A$ extends uniquely to a maximal left-compressed intersecting subfamily of $[n]^{(r)}$.  Moreover, every maximal left-compressed intersecting subfamily of $[n]^{(r)}$ arises in this way.
\end{lemma}

\begin{proof}
 Since $\A$ is left-compressed, it can be completely described by listing its $\lc$-maximal elements $A_1, \ldots, A_k$.  Some of these sets might contain final segments of [2r].  The idea is that the elements of these final segments would take larger values if they were allowed to, so we obtain a generating set by `replacing them by $\infty$'.

For $A = A_i$, take $s$ greatest with $a_s < r+s$ ($s$ exists since $[r+1, 2r]$ is not a member of any left-compressed intersecting family), and let $A' = a_1 \ldots a_s$.  Then $\G = \{A_1', \ldots, A_k'\}$ generates $\A$, as the sets generated by $A_i'$ are precisely those lying below $A_i$.  Since $\G$ is a collection of subsets of $[2r]$ of size at most $r$ and $\A = \gen r {2r} \G$ is intersecting, \cor{new-generators} tells us that $\gen r n \G$ is a left-compressed intersecting family for every $n$.

Now let $\B$ be any extension of $\A$ to a left-compressed intersecting subfamily of $[n]^{(r)}$.  We will show that $\B \subseteq \gen{r}{n}{\G}$.  Indeed, if $\B \not\subseteq \gen r n \G$ then there is a $B \in \B \setminus \gen r n \G$.  We claim that there is a $B' \in [2r]^{(r)}$ with $B' \lc B$ and $B' \not\in \gen r {2r} \G$, contradicting the maximality of $\A$.

We obtain $B'$ from $B$ by compressing as little as possible to get $B' \subseteq [2r]$; that is, we take $B' = (B \cap [2r]) \cup [q, 2r]$ with $q$ chosen such that $|B'| = r$.  Explicitly, $b_i' = \min(b_i, r+i)$.  Now take $G\in\G$.  Since $B \not\in \gen r n \G$, there is an $i$ with $b_i > g_i$.  By construction, $r+i > g_i$.  So $b_i' = \min(b_i, r+i) > g_i$, and $G$ does not generate $B'$.  Hence $\A$ extends uniquely to a maximal left-compressed intersecting subfamily of $[n]^{(r)}$.

It remains to show that every maximal left-compressed intersecting subfamily of $[n]^{(r)}$ arises in this way.  So suppose $\C \subseteq [n]^{(r)}$ is a maximal left-compressed intersecting family with $\C \cap [2r]^{(r)}$ not maximal.  Let $\D_0$ be an extension of $\C \cap [2r]^{(r)}$ to a maximal left-compressed intersecting subfamily of $[2r]^{(r)}$, and let $\D$ be the unique maximal extension of $\D_0$ to $[n]^{(r)}$.  Since $\C$ is maximal and $\D \setminus \C \neq \emptyset$, there is a $C \in \C \setminus \D$.  As above, we obtain $C'\in[2r]^{(r)}$ with $C' \lc C$ and $C' \not\in \D_0$.  But then $C' \not \in \C$, contradicting the assumption that $\C$ is left-compressed.
\end{proof}

\lem{description} allows a compact description of maximal left-compressed intersecting families.  For example, $\{1\}$ generates the star and $\{1(r+1), [2,r+1]\}$ generates the Hilton-Milner family.  Enumerating the generating sets using a computer is feasible for small $r$; for $r=3$ they are $\{1\},\{23\},\{345\},\{14,234\},\{13,235,145\}$ and $\{12,245\}$.

In view of \lem{description}, our key tool is the following.

\begin{lemma}\label{polynomial}
 Let $n \geq 2$, $X \subseteq [2,2r]$.  Then
\[
 |\genX r n \G X| = \sum_{i=1}^r |\genX i {2r} \G X| \binom{n-2r}{r-i}.
\]
\end{lemma}
\begin{proof}
 How do we construct a member of $\genX r n \G X$?  We first choose an initial segment for our set that is contained in $[2r]$ and witnesses the membership of $\genX r n \G X$ (i.e. meets $X$ and is $\genBy$ some $G\in\G$).  We then complete our set by taking as many elements as we need from outside $[2r]$.  This gives rise to the size claimed.
\end{proof}

%%%%%%%%%%%%%%%%%%%%%%%%%%%%%%%%%%%%%%%%%%%%%%%%%%%%%%%%%%%%%%%%%%%%%%%%%%%%%%%%%%%%%%%%%%%%%%%%%%%%%%%

\section{Proof of \thm{main}}\label{main-theorem}

We first show that $X$ is not good if the given conditions do not hold.  We have already seen that for $X \subseteq [2,r+1]$ the Hilton-Milner family shows that $X$ is not good for any $n$.  In each of the remaining cases we claim that the family generated by $\{23\}$ shows that $X$ is not good for any $n$.

So take $X = 23k$ with $k \geq r+2$.  We have
\[
 |\genX r n {\{1\}} {23k}| = \binom {n-2} {r-2} + \binom {n-3}{r-2} + \binom{n-4}{r-2},
\]
where the first term counts the sets containing $1$ and $2$, the second term the sets containing $1$ and $3$ but not $2$, and the third term the sets containing $1$ and $k$ but neither $2$ nor $3$.  Similarly,
\[
 |\genX r n {\{23\}} {23k}| = \binom {n-2} {r-2} + \binom {n-3}{r-2} + \binom{n-3}{r-2},
\]
where the terms count the sets containing $1$ and $2$, the sets containing $1$ and $3$ but not $2$, and the sets containing $2$ and $3$ but not $1$ respectively.  Since $r \geq 3$, $|\genX r n {\{23\}} {23k}| > |\genX r n {\{1\}} {23k}|$ and $23k$ is not good.

Next take $X = 3j$ with $j \geq r+2$.  We have
\[
 |\genX r n {\{1\}} {3j}| = \binom {n-2}{r-2} + \binom{n-3}{r-2},
\]
where the terms count the sets containing $1$ and $3$, and the sets containing $1$ and $j$ but not $3$ respectively.  Similarly,
\[
 |\genX r n {\{23\}} {3j}| = \binom {n-2}{r-2} + \binom{n-3}{r-2} + \binom{n-4}{r-3},
\]
where the terms count the sets containing $1$ and $3$, the sets containing $2$ and $3$ but not $1$, and the sets containing $1$, $2$ and $j$ but not $3$ respectively.  Again, since $r \geq 3$, $|\genX r n {\{23\}} {3j}| > |\genX r n {\{1\}} {3j}|$ and $3j$ is not good.  It follows from \thm{borg}(b) that $2j$ is not good either.

Now we take $X$ satisfying the conditions of the theorem and show that $X$ is good for $n$ sufficiently large.  We will show that, for any $\G \neq \{1\}$, $|\genX  2 {2r} \G X| < |\genX 2 {2r} {\{1\}} X| = |X|$.  Note that, for any $\G$, $|\genX 1 {2r} \G X| = 0$ as the only possible singleton generator is 1, which does not meet $X$.  So by \lem{polynomial}, $\genX 2 n \G X$ has size polynomial in $n$ with leading coefficient $|\genX  2 {2r} \G X|$, from which the result will follow.

There are two maximal left-compressed intersecting families of 2-sets, and $\genX 2 {2r} \G X$ must be contained in one of them.  We handle each case separately.

Suppose first that $\genX 2 {2r} \G X \subseteq \{12, 13, 23\}$. Then it is enough to show that
\[
 |\{12,13,23\}(X)| < |X|.
\]
This is clearly true for $|X| \geq 4$.  If $|X| = 3$, then it is true because one of 2 or 3 is missing from $X$ so that $|\{12,13,23\}(X)| \leq 2$.  If $|X| = 2$, then it is true because both $2$ and $3$ are missing from $X$, so that $|\{12,13,23\}(X)| = 0$.  Finally, if $|X| = 1$, then it is true because $X = \{i\}$ with $i \geq r+2 \geq 4$.

Next suppose that $\genX 2 {2r} \G X \subseteq \{12, 13, \ldots, 1(2r)\}$.  Since $\gen r {2r} \G$ is left-compressed and has a member not containing the element 1, it has $[2,r+1]$ as a member.  Hence by the intersecting property of the generators, $\genX 2 {2r} \G X$ cannot contain $1j$ for any $j \geq r+2$.  But $X \not\subseteq [2,r+1]$, so there is such a $j \in X \setminus [2, r+1]$ and $|\genX 2 {2r} \G X| < |X|$.\qed

\section{Improvements and generalisations}\label{generalisations}

What happens for small $n$?  \thm{borg}(c) tells us that our characterisation cannot be correct for all $n \geq 2r$.
\begin{question}
 How large is `sufficiently large' for $n$ in \thm{main}?
\end{question}
For $2 \leq r \leq 5$, computational results suggest that $n \geq 2r+2$ is large enough for our characterisation to be correct.  It would be particularly nice to show that $n \geq 2r+c$ is sufficient for some constant $c$ independent of $r$.

A natural conjecture is that for $n=2r$, $[2k, 2k+2, \ldots, 2r]$ is the unique minimal good set of its size.  However, this is false; computational results give that $\{7,10\}$ and $\{5,8,10\}$ are unique minimal good sets of their size when $r=5$.
\begin{question}
 Is there a `nice' characterisation of the good sets for $n=2r$ when $r$ is sufficiently large? 
\end{question}
It seems unlikely that a good explicit description exists for intermediate values of $r$ and $n$.  The following may be easier.
\begin{question}
 Is there a short list of families, one of which maximises $|\A(X)|$ for any $X$?
\end{question}

Versions of \lem{description} hold for any property that is preserved under left-compression and can be detected on generating sets.  The most obvious candidate is that of being $t$-intersecting (a family $\A$ is \defn{$t$-intersecting} if $|A \cap B| \geq t$ for all $A, B \in \A$).  Indeed, an identical argument gives the corresponding result that, for large $n$, a set $X \subseteq [t+1, n]$ with $|X| \geq t+3$ is good if and only if $X \not\subseteq [t+1, r+1]$.  (For smaller $X$ the form of good $X$ is again decided by the need to prevent problems caused when $\genX {t+1} {2r-t+1} \G X \subseteq [t+2]^{(t+1)}$.)

In the context of $t$-intersecting families it may be more natural to consider
\[
 \A(s, X) = \{A \in \A:|A \cap X| \geq s\}.
\]
For $s=1$ the argument relies on the fact that maximal left-compressed $t$-intersecting families of $(t+1)$-sets have one of two very simple forms.  For $s=2$, even the $t=1$ case is complicated by the larger number of structures of intersecting families of 3-sets (more generally, $(t+s)$-sets); this problem seems likely to get worse for larger $s$ and $t$.

\vspace{2em}
{\noindent\footnotesize {\bf Acknowledgements.}
I would like to thank the anonymous referees for carefully reading an earlier draft of this paper and making a number of helpful comments.
}

\bibliography{maximum-hitting}{}
\bibliographystyle{plain}

\end{document}